\documentclass{article}

\usepackage{a4}

\usepackage{epsfig}

\usepackage{amsfonts}

\newtheorem{theo}{Theorem}[section]
\newtheorem{prop}[theo]{Proposition}
\newtheorem{corol}[theo]{Corollary}
\newtheorem{lem}[theo]{Lemma}

\newtheorem{defin}[theo]{Definition}
\newtheorem{rem}[theo]{Remark}
\newtheorem{question}[theo]{Question}

\newcommand{\R}{\mathbb R}
\newcommand{\N}{\mathbb N}
\newcommand{\Z}{\mathbb Z}

\newenvironment{proof}{{\flushleft \em Proof : }}{\hfill $\square$ \vspace{5mm}}

\begin{document}

\title{Generic measures for geodesic flows on nonpositively curved manifolds}

\author{Yves Coud\`ene, Barbara Schapira}

\date{21 January 2014}

\maketitle
 
\begin{center}
\emph{Laboratoire de math\'ematiques, UBO,
6 avenue le Gorgeu, 29238 Brest, France}

\emph{LAMFA, UMR CNRS 7352, Universit\'e Picardie Jules Verne, 
      33 rue St Leu, 80000 Amiens, France}
\end{center}

\begin{abstract}
We study the generic invariant probability measures for the geodesic flow 
on connected complete nonpositively curved manifolds.  
Under a mild technical assumption, 
we prove that ergodicity is a generic property in the set of
probability measures defined on the unit tangent bundle of the manifold
and supported by trajectories not bounding a flat strip.
This is done by showing that Dirac measures on periodic orbits
are dense in that set.

In the case of a compact surface, we get the following sharp result: 
ergodicity is a generic property in the space of all invariant measures
defined on the unit tangent bundle of the surface 
if and only if there are no flat strips in the universal cover of the surface.

Finally, we show under suitable assumptions that generically, the invariant
probability measures have zero entropy and are not strongly mixing.
\footnote{37B10, 37D40, 34C28}{}
\end{abstract} 


\section{Introduction} 

Ergodicity is a generic property in the space of probability measures
invariant by a topologically mixing Anosov flow on a compact manifold. 
This result, proven by K. Sigmund in
the seventies \cite{si1},
implies that on a compact connected negatively curved manifold,
the set of ergodic measures is a dense
$G_\delta$ subset of the set of all probability measures invariant by the
geodesic flow. The proof of K. Sigmund's result is based on the specification 
property. This property relies on the uniform hyperbolicity of the system and
on the compactness of the ambient space.

\medskip

\quad
In \cite{CS}, we showed that ergodicity is a generic property of
hyperbolic systems without relying on the specification property.
As a result,  we were able to prove that the set of 
ergodic probability measures invariant by the geodesic flow, on a negatively
curved manifold, is a dense $G_\delta$ set, without any compactness
assumptions or pinching assumptions on the sectional curvature of the
manifold.

\medskip

\quad
A corollary of our result is the existence of ergodic invariant probability
measures of full support for the geodesic flow on any complete negatively
curved manifold, as soon as the flow is transitive. Surprisingly, we succeeded
in extending this corollary to the nonpositively curved setting.
However, the question of genericity in nonpositive curvature appears to be
much more difficult, even for surfaces. In \cite{CS2}, we gave examples of
compact nonpositively curved surfaces with negative Euler characteristic
for which ergodicity is not a generic property in
the space of probability measures invariant by the geodesic flow.  

\medskip

\quad
The first goal of the article is to obtain genericity results in the non
positively curved setting. From now on, all manifolds are assumed to be
\emph{connected, complete} Riemannian manifolds. Recall that a \emph{flat
  strip} in the universal cover of the manifold is a totally geodesic subspace
isometric to the space $[0,r]\times {\bf R}$, for some $r>0$, endowed with its
standard euclidean structure.  
We first show that if there are no flat strip, genericity holds.

\begin{theo}
Let $M$ be a nonpositively curved manifold, 
such that its universal cover has no flat strips. 
Assume that the geodesic flow has at least
three periodic orbit on the unit tangent bundle $T^1M$ of $M$. Then
the set of ergodic probability measures on $T^1M$ is a dense $G_\delta$-subset
 of the set of all probability measures invariant by the flow.
\end{theo}

This theorem is a particular case of theorem \ref{ergodicity-generic} below. 
In the two-dimensional compact case, we get the following sharp result.

\begin{theo}\label{ergodicity-generic-surfaces}
Let $M$ be a nonpositively curved compact orientable surface, with negative
Euler characteristic. Then ergodicity is a generic property in the set
of all invariant probability measures on $T^1M$ if and only if
there are no 
flat strips on the universal cover of $M$.
\end{theo}

In the higher dimensional case, the situation is more complicated.
Under some technical assumption,
we prove that genericity holds in restriction to the set of
nonwandering vectors whose lifts do not bound a flat strip.

\begin{theo}\label{ergodicity-generic} 
Let $M$ be a connected, complete, nonpositively curved manifold, 
and $T^1M$ its unit tangent bundle. 
Denote by $\Omega\subset T^1M$ the nonwandering set of the
geodesic flow, and $\Omega_{NF }\subset \Omega$ the set of nonwandering
vectors that do not bound a flat strip.
Assume that $\Omega_{NF}$ is open in $\Omega$, and contains at least three
different periodic orbits of the geodesic flow. 

Then the set of ergodic probability measures invariant by the 
geodesic flow and with full support in $\Omega_{NF}$ is a $G_\delta$-dense
subset of the set of invariant probability measures on $\Omega_{NF}$.  
\end{theo}

The assumption that $\Omega_{NF}$ is open in $\Omega$ is satisfied
in many examples. For instance, it is true as soon as the number
of flat strips on the manifold is finite.
The set of periodic orbits of the geodesic flow is in $1-1$-correspondence
with the set of oriented closed geodesics on the manifold. Thus, the
assumption that $\Omega_{NF }$ contains at least three different periodic
orbits means that there are at least two distinct nonoriented 
closed geodesics in the manifold that do not lie in the projection 
of a flat strip.
This assumption rules out a few uninteresting examples, such as 
simply connected manifolds or cylinders, and corresponds to the
classical assumption of nonelementaricity in negative curvature.

\medskip 

\quad Whether ergodicity is a generic property in the space of all invariant
measures, in presence of flat strips of intermediate dimension, is still an
open question. In section \ref{examples}, we will see examples with periodic
flat strips of maximal dimension where ergodicity is not generic.

\bigskip

The last part of the article is devoted to mixing and entropy.
Inspired by results of \cite{ABC}, we 
study the genericity of other dynamical properties of measures, 
as zero entropy or mixing. In particular, we prove that

\begin{theo}\label{zero-entropy-and-mixing}
Let $M$ be a connected, complete, nonpositively curved manifold, 
such that $\Omega_{NF}$ contains
at least three different periodic orbits of the geodesic flow and 
is open in the nonwandering set $\Omega$. 

The set of invariant probability measures with zero entropy
for the geodesic flow is generic in the set  of invariant
probability measures on  $\Omega_{NF}$. 
Moreover, the set of invariant probability measures on $\Omega_{NF}$ 
that are not strongly mixing is a generic set. 
\end{theo}

The assumptions in all our results include  the  case where $M$ 
is a noncompact negatively curved manifold. 
In this situation, we have $\Omega=\Omega_{NF}$. 
Even in this case, theorem \ref{zero-entropy-and-mixing} is new. 
When $M$ is a compact negatively curved manifold, it follows 
from \cite{si1}, \cite{Pa2}.
Theorem \ref{ergodicity-generic} was proved in \cite{CS} in the negative
curvature case. 

Results above show that under our assumptions, ergodicity is generic, and
strong mixing is not. We don't know under which condition weak-mixing is a
generic property, except for compact negatively curved manifolds \cite{si1}.
In contrast, topological mixing holds most of the time, and is equivalent to
the non-arithmeticity of the length spectrum (see proposition
\ref{mixing-non-arithmeticity}).

\bigskip 

In section \ref{invariant}, we recall basic facts on nonpositively curved
manifolds and define interesting invariant sets for the geodesic flow. 
In section \ref{Proof-a-la-Cao-Xavier}, we study the case of surfaces.
The next section is devoted to the proof
of theorem  \ref{ergodicity-generic}. 
At last, we prove theorem \ref{zero-entropy-and-mixing} in
sections \ref{section-entropy} and \ref{mixing}. 
 
\bigskip 

During this work, the authors benefited from the ANR grant ANR-JCJC-0108 Geode.


\section{Invariant sets for  the geodesic flow on nonpositively 
curved manifolds}\label{invariant}

Let $M$ be a Riemannian manifold with nonpositive curvature, and let $v$ be a
vector belonging to the unit tangent bundle $T^1M$ of $M$. The vector $v$ is a
\emph{rank one vector}, if the only parallel Jacobi fields along the geodesic
generated by $v$ are proportional to the generator of the geodesic flow. A
connected complete nonpositively curved manifold is a \emph{rank one
manifold} if its tangent bundle admits a rank one vector. In that case, the
set of rank one vectors is an open subset of $T^1M$. Rank one vectors
generating closed geodesics are precisely the hyperbolic periodic points of
the geodesic flow. We refer to the survey of G. Knieper \cite{kni} and the
book of W. Ballmann \cite{ballmann} 
for an overview of the properties of rank one manifolds.

\medskip

Let $X\subset T^1M $ be an invariant set under the action 
of the geodesic flow $(g_t)_{t\in\R}$. 
Recall that the strong stable sets of the flow on $X$ are defined by~:

\smallskip

\noindent
 $W^{ss}(v):=\lbrace \  w\in X \ | \ 
 \lim_{t\rightarrow \infty}
 d(g_{t}(v),g_{t}(w)) = 0 \ \  \rbrace$ ; 
 \\
 $W^{ss}_{\varepsilon}(v):=\lbrace \  w\in W^{ss}(v) \ | \ 
 d(g_{t}(v),g_{t}(w))\leq \varepsilon$ for all $t\ge 0 
 \ \  \rbrace$.

\smallskip

One also defines the strong unstable sets $W^{su}$ and
$W^{su}_{\varepsilon}$ of $g_t$ ; these are the stable sets of
$g_{-t}$.

\medskip

Denote by $\Omega\subset T^1 M$ the 
nonwandering set of the geodesic flow, that
is the set of vectors $v\in T^1M$ such that 
for all neighbourhoods $V$ of $v$,
there is a sequence $t_n\to \infty$ , 
with $g^{t_n}V\cap V\neq \emptyset$. 
Let us introduce several interesting invariant subsets
of the nonwandering set $\Omega$ of the geodesic flow. 

\begin{defin}Let $v\in T^1M$.
We say that {\em its strong stable (resp. unstable) manifold coincides 
with its strong stable (resp. unstable) horosphere} if, for any lift 
$\tilde{v}\in T^1\tilde{M}$ of $v$,  
for all $\tilde{w}\in T^1\tilde{M}$, 
the existence of a constant $C>0$ s.t.
 $d(g^t\tilde{v}, g^t\tilde{w})\le C$ for
 all $t\ge 0$ 
(resp. $t\le 0$) implies that there exists $\tau\in\R$ such that 
$d(g^tg^\tau\tilde{v}, g^t\tilde{w})\to 0$ 
when $t\to +\infty$ 
(resp. $t\to -\infty$). 
\end{defin}

Denote by $T_{hyp}^+\subset T^1M$ (resp. $T_{hyp}^-$) the set of vectors whose
stable (resp. unstable) manifold coincides with its stable (resp. unstable)
horosphere, $T_{hyp}=T_{hyp}^+\cap T_{hyp}^-$ and $\Omega_{hyp}=\Omega\cap
T_{hyp}$.

The terminology comes from the fact that on $\Omega_{hyp}$, a lot of
properties of a hyperbolic flow still hold. However, periodic orbits in
$\Omega_{hyp}$ are not necessarily hyperbolic in the sense that they can have
zero Lyapounov exponents, for example higher rank periodic vectors.

\begin{defin} Let $v\in T^1M$. 
We say that {\em $v$ does not bound a flat strip}
if no lift $\tilde{v}\in T^1\tilde{M}$ of $v$ determines a geodesic
which bounds an infinite flat (euclidean) strip isometric 
to $[0,r]\times\R$, $r>0$, on $T^1\tilde{M}$.

The projection of a flat strip on the manifold $M$ 
is called a \emph{periodic flat strip} if it contains a periodic geodesic. 

We say that {\em $v$ is not contained in a periodic flat strip} 
if the geodesic determined by $v$  on $M$  does not stay 
in a periodic flat strip for all $t\in\R$.  
\end{defin}

In \cite{CS}, we restricted the study of the dynamics to the set $\Omega_1$ of
nonwandering rank one vectors whose stable (resp. unstable) manifold coincides
with the stable (resp. unstable) horosphere. If $\mathcal{R}_1$ denotes 
the set of rank one vectors, then $\Omega_1=\Omega_{hyp}\cap \mathcal{R}_1$. 
The dynamics on $\Omega_1$ is very close from the dynamics
of the geodesic flow on a negatively curved manifold, 
but this set is not very natural, and too small in general.
We improve below our previous results,
by considering the following larger sets: 
\begin{itemize}
\item[$\bullet$]
the set $\Omega_{NF}$ of nonwandering vectors that do not bound a flat strip, 
\item[$\bullet$] 
the set $\Omega_{NFP}$ of nonwandering vectors that are not contained in 
a periodic flat strip, 
\item[$\bullet$] the set $\Omega_{hyp}$ of nonwandering vectors 
whose stable (resp. unstable) manifold coincides with the stable horosphere. 
\end{itemize}

We have the inclusions 
$$
\Omega_1\subset\Omega_{hyp}\subset 
\Omega_{NF}\subset\Omega_{NFP}\subset \Omega\,,
$$ 
and they can be strict, except if $M$ has negative curvature, in which case
they all coincide. Indeed, a higher rank periodic vector is not in $\Omega_1$,
but it can be in $\Omega_{hyp}$ when it does not bound a flat strip of
positive width. A rank one vector whose geodesic is asymptotic to a flat
cylinder is in $\Omega_{NF}$ but not in $\Omega_{hyp}$. 

\begin{question}\rm 
It would be interesting to understand when we have the equality 
$\Omega_{NF}=\Omega_{NFP}$.
We will show that  on 
compact rank one surfaces, if there is a flat strip, then there exists also a periodic flat strip. 
When the surface is a flat torus, 
 we have of course $\Omega_{NF}=\Omega_{NFP}=\emptyset$. 

It could also  happen on some noncompact rank one manifolds 
that all vectors that bound a nonperiodic flat
 strip are wandering, so that $\Omega_{NF}=\Omega_{NFP}$. 

Is it true on all rank-one surfaces, and/or all rank-one compact manifolds, 
that $\Omega_{NF}=\Omega_{NFP}$~? 
\end{question}

In the negative curvature case, it is standard to assume  the fundamental group of $M$ to be \emph{nonelementary}. This means
that there exists at least two (and therefore infinitely many) 
closed geodesics on $M$, and therefore at least four (and in fact 
infinitely many) periodic orbits of the geodesic flow
on $T^1M$ (each closed geodesic lifts to $T^1M$ into
two periodic curves, one for each orientation). 
This allows to discard simply connected manifolds or hyperbolic cylinders,
for which there is no interesting recurring dynamics.

In the nonpositively curved case, we must also get rid of 
flat euclidean cylinders, for which there are infinitely many
periodic orbits, but no other recurrent trajectories.
So we will assume that there exist at least three different periodic
orbits in $\Omega_{NF}$, that is, two distinct closed
geodesics on $M$ that do not bound a flat strip. 

\medskip

We will need another stronger assumption, on the flats of the manifold. 
{\em To avoid to deal with flat strips, we will  
 work in restriction to $\Omega_{NF}$, 
with the additional assumption that $\Omega_{NF}$ is open in $\Omega$.}
This is satisfied for example if  $M$ admits
only finitely many flat strips.
We will see that this assumption insures that the periodic orbits
that do not bound a flat strip are dense in $\Omega_{hyp}$ and $\Omega_{NF}$.

In the proof of theorems  \ref{ergodicity-generic}
and \ref{zero-entropy-and-mixing}, the key step is the proposition below.

\begin{prop}\label{densite-orbites-periodiques}
Let $M$ be a connected, complete, nonpositively curved manifold, 
which admits at least three different periodic orbits that do not
bound a flat strip. Assume 
that $\Omega_{NF}$ is open in $\Omega$. 
Then the Dirac measures supported by the periodic orbits
of the geodesic flow $(g^t)_{t\in\R}$  that are in $\Omega_{NF}$, 
are dense in the set of all invariant probability measures defined 
on $\Omega_{NF}$.  
\end{prop}  

\section{The case of surfaces} \label{Proof-a-la-Cao-Xavier}

In this section, $M$ is a compact, connected, nonpositively curved
orientable surface. We prove theorem \ref{ergodicity-generic-surfaces}. 
 
If the surface admits a periodic flat strip, by our results in \cite{CS2}, 
we know that ergodicity cannot be generic. 
In particular, a periodic orbit in the middle of the flat strip is not 
in the closure of any ergodic invariant probability measure of full support. 

If the surface admits no flat strip, then $\Omega=\Omega_{NF}=T^1M$, 
so that the result follows from theorem \ref{ergodicity-generic}. 
It remains to show the following result.

\begin{prop}\label{nonperiodic-flatstrip} 
Let $M$ be a compact connected orientable nonpositively curved surface. 
If it admits a nonperiodic flat strip, 
then it admits also a periodic flat strip. 
\end{prop}

The proof is inspired by unpublished work of Cao and Xavier. 
We proceed by contradiction. Assume that $M$ is not a flat torus, 
but  admits a nonperiodic flat strip. 

There is an isometric embedding from $I\times \R_+$ to the universal cover
$\widetilde{M}$. The interval $I$ is necessarily bounded. Indeed, the manifold
$M$ is compact so it admits a relatively compact connected fundamental domain
for the action of its fundamental group on $\widetilde{M}$. Such a fundamental
domain cannot be completely included in a flat strip with infinite width.

Let $R_{max}\in \R_+$ be the supremum of the widths $r$ 
of nonperiodic  flat strips $[0,r]\times\R_+$
of $\widetilde{M}$. The above argument shows that $R_{max}$ is finite. 

Consider a vector $\tilde{v}\in T^1\widetilde{M}$ generating a trajectory
$(g_t \tilde{v})_{t\ge 0}$ on $T^1\widetilde{M}$ that is tangent to a
nonperiodic flat strip $[0,R]\times \R$, for some $\frac{9}{10} R_{max}\le
R\le R_{max}$. Assume that $R$ is maximal among the width of all flat
strips containing $\tilde{v}$, and relocate $\tilde{v}$ on the boundary
of the strip. Assume also that the trajectory $(g_t \tilde{v})_{t\ge 0}$ 
bounds the right side $\{R\}\times \R_+$ of the flat strip and denote by $v$
the image of $\tilde{v}$ on $T^1M$.

Since $M$ is compact, we can assume that there is a subsequence $g_{t_n} v$,
with $t_n\to +\infty$, such that $g_{t_n}v$ converges to some vector
$v_\infty$. This vector also lies on a flat strip of width at least $R$.
Indeed, consider a lift $\tilde{v}_\infty$ of ${v}_\infty$ and isometries
$\gamma_n$ of $\tilde{M}$ such that $\gamma_n(g_{t_n}\tilde{v})$ converges to
$\tilde{v}_\infty$. Every point on the half-ball of radius $R$ centered on the
base point of $\tilde{v}_\infty$ is accumulated by points on the euclidean
half-balls centered on $\gamma_n(g_{t_n}\tilde{v})$, so the curvature vanishes
on that half-ball. We can talk about the segment in the half-ball starting
from the base point of $\tilde{v}_\infty$ and orthogonal to the trajectory of
$\tilde{v}_\infty$. Vectors based on that segment and parallel to
$\tilde{v}_\infty$ are accumulated by vectors generating geodesics in the flat
strips bounding $\gamma_n(g_{t_n}\tilde{v})$. Hence the curvature vanishes
along the geodesics starting from these vectors and we get a flat strip of
width at least $R$.

If $v_\infty$ lies on a periodic flat strip, the proof is finished. 
Assume  therefore that the flat strip of $v_\infty$ is not periodic. 

The idea  is now  to use the nonperiodic flat strip of $v$ to
construct a flat strip of width strictly larger than $R_{max}$. 
By definition of $R_{max}$, this new flat strip is necessarily periodic, 
and we get the desired result.

The vectors $g_{t_n}v$ converges to $v_\infty$, so consider $t>0$ 
so that the base point of $g_tv$ is very close to the base point of
$v_{\infty}$ and the image of $g_tv$ by the parallel transport from
$T^1_{\pi(g_tv)}M$ to $T^1_{\pi(v)}M$ makes a small angle $\theta$ with $v$.
Observe that this angle $\theta$ is nonzero. Indeed, otherwise, 
the flat strips bounded by $\gamma_{n}(g_{t_n}\tilde{v})$ and $\tilde{v}_\infty$
would be parallel. The flat strip bounded by $\tilde{v}_\infty$ would
extend the flat strip bounded by $\gamma_n(g_{t_n}\tilde{v})$ by a quantity
roughly equal to the distance between their base points, ensuring that the
flat strip bounded by $\tilde{v}$ is actually larger than $R$ and
contradicting the fact that $R$ is the width of this flat strip.
 
\medskip 

Let us now consider a time $t$ such that $g_tv$ is close to $v_\infty$, 
with the angle between these two vectors denoted by $\theta>0$. 
When the flat strip comes back close to $v_\infty$ at time $t$, it cuts
the boundary of the flat strip  along a segment whose length is denoted by $L$.
Without loss of generality, we may assume
that $v_\infty$ lies on the right boundary of its flat strip. 
Let us consider the highest rectangle of length $L/2$ that we can
put at the end of this segment, on its right side, and that belongs
to the returning flat strip of $g_t(v)$ but not to the flat strip 
of $v_\infty$. This rectangle is pictured below, its width is denoted by $H$.

\begin{center}
\epsfig{figure=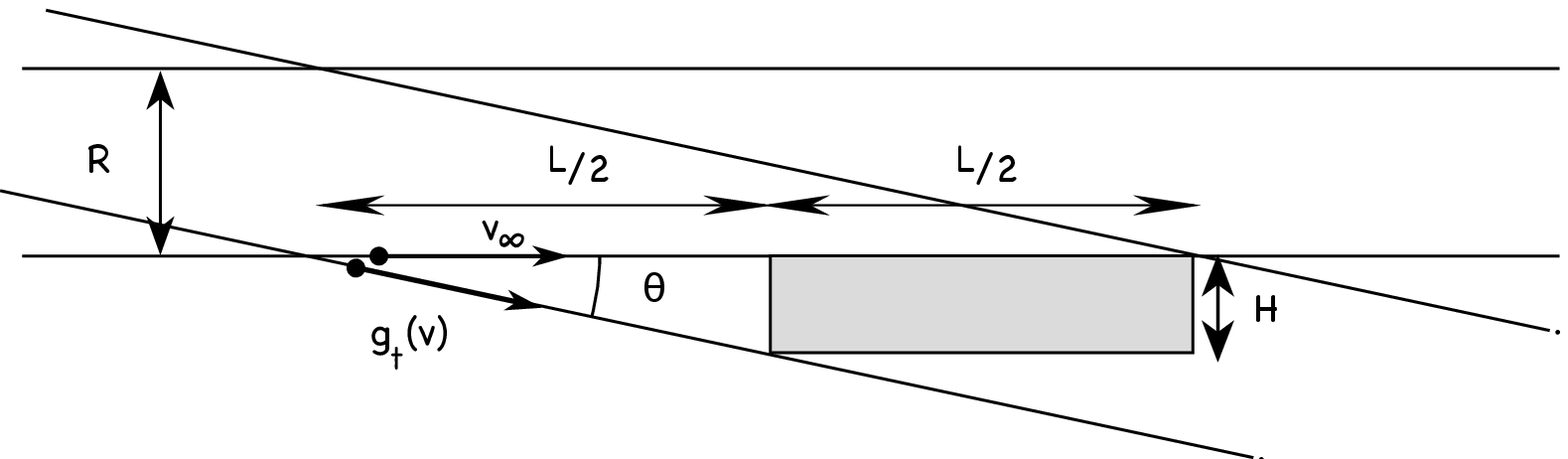,width=.9\textwidth,angle=0}
\end{center}

\noindent
The quantities $H$ and $L$ can be computed 
using elementary euclidean trigonometry.
$$
H =  {R \over 2 \cos \theta} \geq {R\over 2}
$$
$$
L = {R \over \sin \theta} 
\ \oalign{\hbox to 3.1em{\rightarrowfill}\cr
$\scriptstyle\ \theta\rightarrow 0$ \cr}\ +\infty
$$
The same computation works when $v_\infty$ is not on the boundary 
of its flat strip. The rectangle has a width bounded from below and a length 
going to infinity when $\theta$ goes to zero. 

\medskip

Every time the trajectory of $v$ comes back near $v_\infty$,
we get a new rectangle, and this gives
a sequence of rectangles of increasing length right next to the flat strip.
The next picture shows these rectangles in the universal cover $\widetilde{M}$.

\begin{center}
\epsfig{figure=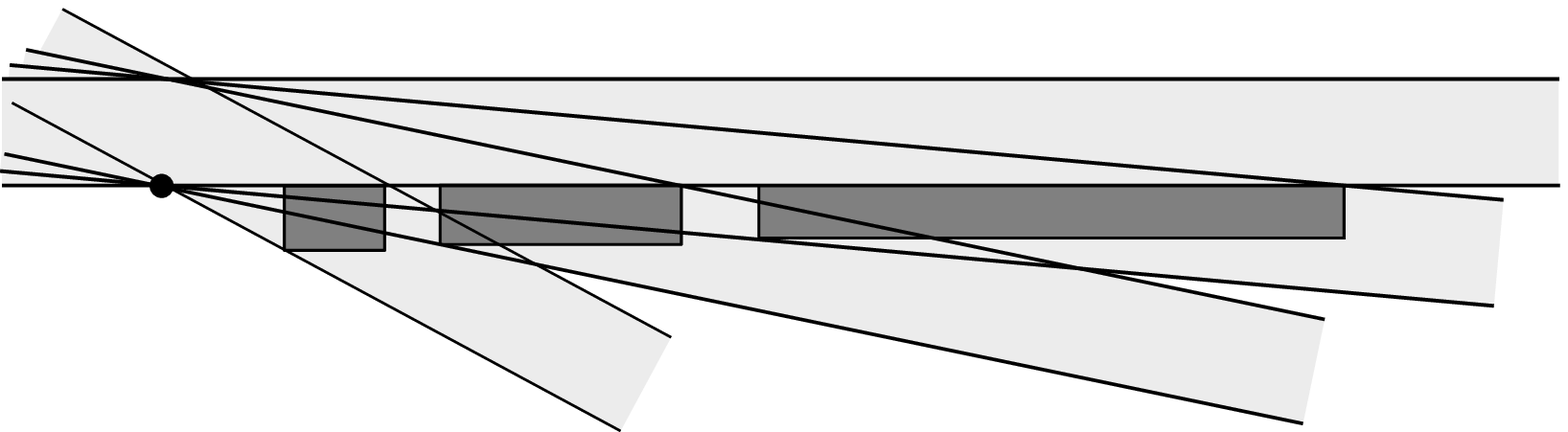,width=.9\textwidth,angle=0}
\end{center}

These rectangles are tangent to the flat strip of width $R$ bounded by
${v}_{\infty}$, so that we get a sequence of flat rectangles of width
$3R/2$ and length $L$ going to infinity. By compactness of $M$, this sequence
accumulates to some infinite flat strip of width $3R/2>R_{max}$. Therefore
this flat strip is periodic, by definition of $R_{max}$. This ends the proof.


\section{The density of Dirac measures in ${\cal M}^1(\Omega_{NF})$ } \label{section4}

This section is devoted to the proof of
proposition \ref{densite-orbites-periodiques} 
and theorem  \ref{ergodicity-generic}.

\subsection{Closing lemma, local product structure and transitivity}

Let $X$ be a metric space, and $(\phi^t)_{t\in\R}$ be a continuous flow acting
on $X$. In this section, we recall three fundamental dynamical properties
that we use in the sequel: the closing lemma, the local product structure, 
and transitivity.

\medskip

When these three properties are satisfied on $X$, 
we proved in \cite{CS}  (prop. 3.2 and corollary 2.3) that 
the conclusion of proposition \ref{densite-orbites-periodiques} 
holds on $X$:   the invariant probability measures
supported by periodic orbits are dense 
in the set of all Borel invariant probability measures on $X$. 

In \cite{Pa1}, Parthasarathy notes that the density of Dirac measures on
periodic orbits is important to understand the dynamical properties of the
invariant probability measures, and he asks under which assumptions it is
satisfied. 
In the next sections, we will prove weakened versions of these three properties 
(closing lemma, local product and transitivity), 
and deduce proposition \ref{densite-orbites-periodiques}.

\begin{defin}\label{def-closing_lemma} A flow $\phi_t$ on a metric space $X$ 
satisfies the {\em closing
 lemma} if for all points $v\in X$, and $\varepsilon>0$, 
there exist  a neighbourhood
 $V$ of $v$ ,   $\delta>0$ and a $t_0>0$
 such that for all $w\in V$ and all $t>t_0$ with $d(w,\phi_t w)<\delta$ and 
 $\phi_t w\in V$,
 there exists $p_0$ and $l>0$, with  $|l-t|<\varepsilon$,
 $\phi_{l}p_0=p_0$, and $d(\phi_s p_0,\phi_s w)<\varepsilon$
 for $0<s<\min(t,l)$.
\end{defin}

\begin{defin}\label{def-local_product_structure} 
The flow ${\phi}_t$ is said to admit a 
{\em local product structure} if all points $u\in X$ have   
a neighbourhood $V$ which satisfies : for all
$\varepsilon>0$, there exists a positive constant $\delta$,
such that for all $v,w\in V$ with  $d(v,w)\leq \delta$, 
there is a point $<v,w>\in X$, 
a real number $t$ with $|t|\leq \varepsilon$, so that: 
$$<v,w> \in {W}^{su}_{\varepsilon}\bigl({\phi}_{t}(v)\bigr)\cap 
{W}^{ss}_{\varepsilon}(w).$$
\end{defin}

\begin{defin}\label{def-transitivity} 
The flow $(\phi^t)_{t\in\R}$ is {\em transitive} if for all non-empty
open sets $U$ and $V$ of $X$, and $T>0$, there is $t\ge T$ such that 
$\phi^t(U)\cap V\neq \emptyset$. 
 \end{defin}

Recall that if $X$ is a $G_\delta$ subset of a
complete separable metric space, then it is a Polish space, and the set
$\mathcal{M}^1(X)$ of invariant probability measures on $X$ is also a Polish
space. As a result, the Baire theorem holds on this space \cite{bi}
th 6.8. In particular, this will be the case for the set $X=\Omega_{NF}$ 
when it is open in $\Omega$, since $\Omega$ is a closed subset of $T^1M$.

\medskip

If $M$ is negatively curved, we saw in \cite{CS} that the restriction 
of $(g_t)_{t\in\R}$ to $\Omega$ 
satisfies the closing lemma, the local product structure, and is transitive. 
Note that we do not need any (lower or upper) bound  on the curvature, i.e. 
we allow the curvature to go to $0$ or to $-\infty$
in some noncompact parts of $M$. In particular, the conclusions of all 
theorems of this article   apply to  the geodesic flow on 
the nonwandering set of any nonelementary negatively curved manifold. 


\subsection{Closing lemma and transitivity on $\Omega_{NF}$} 

We start by a proposition essentially due to G. Knieper 
(\cite{Knieper} prop 4.1).

\begin{prop}\label{recurrents-sont-bons}
Let $v\in \Omega_{NF}$ be a recurrent 
vector which does not bound a flat strip. Then $v\in \Omega_{hyp}$, i.e. 
its strong stable (resp. unstable) manifold coincides 
with its stable (resp. unstable) horosphere. 
\end{prop}  

\begin{proof} Let $\widetilde{M}$ the universal cover of $M$ and
 $\tilde{v}\in T^1\widetilde{M}$ be a lift of $v$. 
Assume that there exists $w\in T^1M$ which belongs to the stable horosphere, 
but not to the strong stable manifold of $v$.  
We can therefore find $c>0$,
 such that $0<c\le d(g^t\tilde{v},g^t\tilde{w})\le d(v,w)$, for all $t\ge 0$. 
Let us denote by $\Gamma$ the deck transformation group of the 
covering $\tilde{M}\rightarrow M$.
This group acts by isometries on $T^1\tilde{M}$.
The vector $v$ is recurrent, so there exists $\gamma_n\in \Gamma$, 
$t_n\to\infty$, with $\gamma_n(g^{t_n}\tilde{v})\to\tilde{v}$. 
Therefore, for all $s\ge -t_n$, 
we have $c\le d(g^{t_n+s}\tilde{v},g^{t_n+s}\tilde{w})=
d(g^s\gamma_ng^{t_n}v,g^s\gamma_n g^{t_n}w)\le d(v,w)$.
Up to a subsequence, we can assume that $\gamma_n g^{t_n}w$ 
converges to a vector $z$. 
Then we have for all $s\in \R$, $0<c\le d(g^s \tilde{v}, g^s z)\le d(v,w)$. 
The flat strip theorem shows that $\tilde{v}$ bounds a flat strip
(see e.g. \cite{ballmann} cor 5.8).
This concludes the proof. 
\end{proof}


In order to state the next result, we recall a definition.
The ideal boundary of the universal cover, denoted by $\partial\tilde{M}$,
is the set of equivalent classes of half geodesics that stay at a bounded
distance of each other, for all positive $t$. We note $u_+$ the class
associated to the geodesic $t\mapsto u(t)$, and $u_-$ the class
associated to the geodesic $t\mapsto u(-t)$. 

\begin{lem}[Weak local product structure]\label{weak-local-product} 
Let $M$ be a complete, connected, nonpositively curved manifold, 
and $v_0$ be a vector that does not bound a flat strip. 
\begin{enumerate}
\item  For all $\varepsilon>0$, there exists 
$\delta>0$, such that if $v,w\in T^1M$ satisfy $d(v,v_0)\le\delta$, 
$d(w,v_0)\le \delta$, there exists a vector $u=<v,w>$ 
satisfying $u^-=v^-$, $u^+=v^+$, and $d(u,v_0)\le \varepsilon$. 
\item Moreover, if $v,w\in T_{hyp}$,
then  $u=<v,w>\in T_{hyp}$.  
\end{enumerate}
\end{lem}

\noindent
This lemma will be applied later 
to recurrent vectors that do not bound a flat strip;
these are all in $\Omega_{hyp}$. 

\begin{proof}
The first item of this lemma is an immediate reformulation
of \cite{ballmann} lemma 3.1 page 50. 
The second item comes from the definition of the set $T_{hyp}$
of vectors whose stable (resp. unstable) 
manifold coincide with the stable (resp. unstable) horosphere. 
\end{proof}

Note that \emph{a priori}, the local product structure as stated in definition 
\ref{def-local_product_structure} and in \cite{CS} 
is not satisfied on $\Omega_{NF}$: if $v,w$ are in $\Omega_{NF}$, 
the local product $<v,w>$ 
does not necessarily belong to $\Omega_{NF}$.


\begin{lem}\label{weak-closing-lemma}
Let $M$ be a nonpositively curved manifold such that 
$\Omega_{NF}$ is open in $\Omega$. 
Then the closing lemma (see definition \ref{def-closing_lemma})
is satisfied in restriction to $\Omega_{NF}$. 
\end{lem} 

\begin{proof}
We  adapt the  argument of  Eberlein \cite{eb1} 
(see also  the proof of theorem 7.1 in \cite{CS}). 
Let $u\in \Omega_{NF}$, $\varepsilon>0$ and $U$ be a neighborhood
of $u$ in $\Omega$. We can assume that
$U\subset\Omega_{NF}\subset\Omega$
since $\Omega_{NF}$ is open in $\Omega$.
Given $v\in U\cap \Omega_{NF}$, with $d(g^tv,v)$ very small 
for some large $t$, it is enough to find a periodic orbit $p_0\in U$ 
shadowing the orbit of $v$ during a time $t\pm \varepsilon$. 
Since the sets $\Omega_{hyp}$ and $\Omega_{NF}$ 
have the same periodic orbits, we will deduce that 
$p_0\in\Omega_{hyp}\subset\Omega_{NF}$. 

Choose $\varepsilon>0$, and assume by contradiction 
that there exists a sequence $(v_n)$ in $\Omega_{NF}$, $v_n\to u$, 
and $t_n\to +\infty$, such that 
$d(v_n,g^{t_n}v_n)\to 0$, with no periodic orbit of length 
approximatively $t_n$ shadowing the orbit of $v_n$. 

Lift everything to $T^1\tilde{M}$. There exists $\varepsilon>0$, 
$\tilde{v}_n\to \tilde{u}$, $t_n\to +\infty$, 
and a sequence of isometries $\varphi_n$ of $\widetilde{M}$ s.t. 
$d(\tilde{v}_n, d\varphi_n\circ g^{t_n}\tilde{v}_n)\to 0$. 
Now, we will show that for $n$ large enough, $\varphi_n$ is an 
axial isometry, and find on  its axis a 
vector $\tilde{p}_n$ which is the lift of a periodic orbit of 
length $\omega_n=t_n\pm\varepsilon$ shadowing the orbit of $v_n$. 
This will conclude the proof by contradiction. 

Let $\gamma_{\tilde{u}}$ be the geodesic determined by 
$\tilde{u}$, and $u^\pm$ its endpoints at infinity, 
$x\in\tilde{M}$ (resp. $x_n$, $y_n$) the 
basepoint of $\tilde{u}$ (resp. $\tilde{v}_n$, $g^{t_n}\tilde{v}_n$). 
As $\tilde{v}_n\to \tilde{u}$, $t_n\to +\infty$, 
$x_n\to x$, and $d(\varphi_n^{-1}(x_n),y_n)\to 0$,
we see easily that $\varphi_n^{-1}(x)\to u^+$. Similary, $\varphi_n(x)\to u^-$. 

Since $\tilde{u}$ does not bound a flat strip, Lemma 3.1 
of \cite{ballmann} implies that for all $\alpha>0$, 
there exist neighbourhoods $V_\alpha(u^-)$ and $V_\alpha(u^+)$ 
of $u^-$ and $u^+$ respectively, in the boundary at 
infinity of $\tilde{M}$, such that for all $\xi^-\in V_\alpha(u^-)$ 
and $\xi^+ \in V_\alpha(u^+)$, 
there exists a geodesic joining $\xi^-$ and $\xi^+$ and 
at distance less than $\alpha$ from $x=\gamma_{\tilde{u}}(0)$. 

Choose $\alpha=\varepsilon/2$.
We have $\varphi_n(x)\to u^-$ and $\varphi_n^{-1}(x)\to u^+$, so for
$n$ large enough, 
$\varphi_n(V_{\varepsilon/2}(u^-))\subset V_{\varepsilon/2}(u^-)$ 
and $\varphi_n^{-1}(V_{\varepsilon/2}(u^+))\subset V_{\varepsilon/2}(u^+)$. 
By a fixed point argument, we find two fixed points 
$\xi_n^\pm\in V_{\varepsilon/2}(u^\pm)$ of $\varphi_n$, 
so that $\varphi_n$ is an axial isometry. 

Consider the geodesic joining $\xi_n^-$ to $\xi_n^+$ given 
by W. Ballmann's lemma. 
It is invariant by $\varphi_n$, which acts by translation on it, 
so that it induces on $M$ a periodic geodesic, and on 
$T^1M$ a periodic orbit of the geodesic flow. 
Let $p_n$ be the vector of this orbit minimizing the distance
 to $u$, and $\omega_n$ its period. 
The vector $p_n$ is therefore close to $v_n$, and its period close to $t_n$, 
because $d\varphi_n^{-1}(\tilde{p}_n)=g^{\omega_n}\tilde{p}_n$ 
projects on $T^1M$ to $p_n$, 
 $d\varphi_n^{-1}(\tilde{v}_n)=g^{t_n}\tilde{v}_n$ projects to
 $g^{t_n}v_n$, $d(g^{t_n}v_n,v_n)$ is small,
and $\varphi_n$ is an isometry.  
Thus, we get the desired contradiction.
\end{proof}

\begin{lem}[Transitivity]\label{transitivity} 
Let $M$ be a connected, complete, nonpositively curved manifold
which contains at least three distinct periodic orbits
that do not bound a flat strip.
If $\Omega_{NF}$ is open in $\Omega$, 
then the restriction of the  geodesic flow to any of the two sets
$\Omega_{NF}$ or $\Omega_{hyp}$ is transitive. 
\end{lem}

Transitivity of the geodesic flow on $\Omega$
was already known under the so-called duality condition, 
which is equivalent to the equality $\Omega=T^1M$ 
(see \cite{ballmann} for details and references). 
In that case, $\Omega_{hyp}$ is dense in $T^1M$.

\begin{proof} 
Let $U_1$ and $U_2$ be two open sets in $\Omega_{NF}$.
Let us show that there is a trajectory in $\Omega_{NF}$ that starts from
$U_1$ and ends in $U_2$. This will prove transitivity on $\Omega_{NF}$.

The closing lemma implies that periodic orbits 
in $\Omega_{hyp}$ are dense in $\Omega_{NF}$ and $\Omega_{hyp}$. 
So we can find two periodic vectors 
$v_1$ in $\Omega_{hyp}\cap U_1$, and  $v_2$ in $\Omega_{hyp}\cap U_2$. 
Let us assume that $v_2$ is not
opposite to $v_1$ or an iterate of $v_1$:
$-v_2\not\in \cup_{t\in {\bf R}} g_t(\{v_1\})$. 
Then there is a vector $v_3\in T^1M$
whose trajectory is negatively asymptotic to the trajectory of 
$v_1$ and positively asymptotic to
the trajectory of $v_2$, cf \cite{ballmann} lemma 3.3. 
Since $v_1$ and $v_2$ are in $\Omega_{hyp}$, 
the vector $v_3$ also belongs to $T_{hyp}$, 
and therefore does not bound a periodic flat strip. 

\medskip

Let us show that $v_3$ is nonwandering.
First note that there is also a trajectory negatively asymptotic to 
the negative trajectory of $v_2$ and positively asymptotic to
the trajectory of $v_1$. That is, the two periodic orbits $v_1$, $v_2$
are connected as pictured below.


\begin{center}
\epsfig{figure=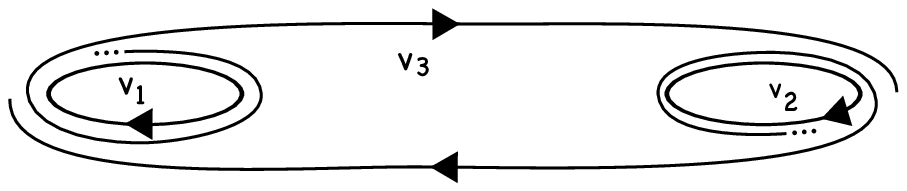,width=0.4\textwidth,angle=0}
\end{center}

This implies that the two connecting orbits are nonwandering:
indeed, using the local product structure, we can glue the two
connecting orbits to obtain a trajectory that starts close to $v_3$,
follows the second connecting orbit, and then follows the orbit of
$v_3$, coming back to the vector $v_3$ itself. Hence $v_3$ is in
$\Omega$. Since it is in $T_{hyp}$ 
it belongs to $\Omega_{hyp}\subset\Omega_{NF}$
and we are done.

\medskip

If $v_1$ and $v_2$ generate opposite trajectories, then we take a third
periodic vector $w$ that does not bound a flat strip,
and connect first $v_1$ to $w$ then $w$ to $v_2$.
Using again the product structure, we can glue the connecting
orbits to create a nonwandering trajectory from $U_1$ to $U_2$.
\end{proof}

\begin{rem}\label{transitivity-remark}\rm We note 
that without any topological assumption on $\Omega_{NF}$,
the same argument gives transitivity of the geodesic flow 
on the closure of the set of periodic hyperbolic vectors.  
\end{rem}


\subsection{Density of Dirac measures on periodic orbits } \label{proofs}

Let us now prove proposition \ref{densite-orbites-periodiques}, that states
the following:

\smallskip 

\noindent\emph{%
Let $M$ be a connected, complete, nonpositively curved manifold, 
which admits at least three different periodic orbits that do not
bound a flat strip. Assume
that $\Omega_{NF}$ is open in $\Omega$. 
Then the Dirac measures supported by the periodic orbits
of the geodesic flow $(g^t)_{t\in\R}$  that are in $\Omega_{NF}$, 
are dense in the set of all invariant probability measures defined 
on $\Omega_{NF}$.  }

\begin{proof} 
We first show that Dirac measures on periodic orbits not bounding
a flat strip are dense in the set of ergodic invariant probability measures
on $\Omega_{NF}$. 

Let $\mu$ be an ergodic invariant probability measure supported
by $\Omega_{NF}$. By Poincar\'e and Birkhoff theorems, 
$\mu$-almost all vectors are recurrent and generic w.r.t. $\mu$. 
Let $v\in\Omega_{NF}$ be such a recurrent generic vector w.r.t. $\mu$
that belongs to $\Omega_{NF}$. The closing lemma \ref{weak-closing-lemma}
gives a periodic orbit close to $v$.
Since $\Omega_{NF}$ is open in $\Omega$,
that periodic orbit is in fact in $\Omega_{NF}$.
The Dirac measure on that orbit is close to $\mu$ and the claim is proven.

\medskip 

The set $\mathcal{M}^1(\Omega)$ is the convex hull of the 
set of invariant ergodic probability measures, 
so the set of convex combinations of periodic
measures not bounding a flat strip is dense in the set of 
all invariant probability measures on $\Omega_{NF}$. 
It is therefore enough to prove that periodic
measures not bounding a flat strip are dense in the set of convex
combinations of such measures.
The argument follows \cite{CS}, with some subtle differences.

\medskip

Let $x_1$, $x_3$, ..., $x_{2n-1}$ be periodic vectors of $\Omega_{NF}$
with periods $l_1$, $l_3$,..., $l_{2n-1}$, and $c_1$, $c_3$,..., $c_{2n-1}$
positive real numbers with $\Sigma\, c_{2i+1}=1$. Let us denote the
Dirac measure on the orbit of a periodic vector $p$ by 
$\delta_p$. We want to find a periodic vector $p$ such that 
$\delta_p$ is close to the sum $\Sigma\, c_{2i+1}\,\delta_{x_{2i+1}}$.
The numbers $c_{2i+1}$ may be assumed to be rational numbers of the 
form $p_{2i+1}/q$. Recall that the $x_i$ are in fact in $\Omega_{hyp}$.

\smallskip

The flow is transitive on $\Omega_{NF}$ (lemma \ref{transitivity}),
hence for all $i$, there is a vector $x_{2i}\in \Omega_{NF}$
close to $x_{2i-1}$ whose trajectory becomes close to $x_{2i+1}$,
say, after time $t_{2i}$. We can also find a point $x_{2n}$ close to
$x_{2n-1}$ whose trajectory becomes close to $x_1$ after some time.
The proof of lemma \ref{transitivity} actually tells us that the
$x_{2i}$ can be chosen in $\Omega_{hyp}$.

\medskip


\begin{center}
\epsfig{figure=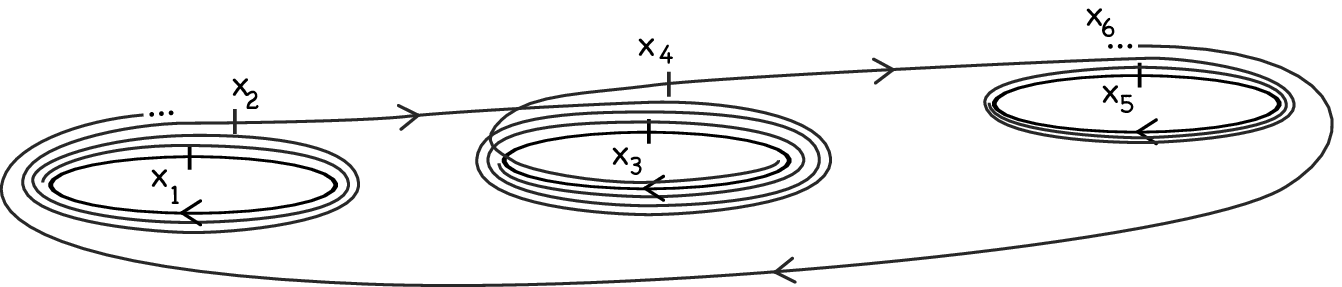,width=0.8\textwidth,angle=0}
\end{center}

\medskip

Now these trajectories can be  glued together, 
using the local product on $\Omega_{hyp}$  
(lemma \ref{weak-local-product}) in the neighbourhood 
of each $x_{2i+1}\in \Omega_{hyp}$, as follows:
we fix an integer $N$, large enough. 
First glue the piece of periodic orbit starting
from $x_1$, of length $Nl_1p_1$, together with the orbit of $x_2$, of
length $t_2$. The resulting orbit ends in a neighbourhood of $x_3$, and
that neighbourhood does not depend on the value of $N$. This orbit is
glued with the trajectory starting from $x_3$, of length $Nl_2p_2$,
and so on (See \cite{cou} for details). 

\medskip

We end up with a vector close to $x_1$, whose trajectory is
negatively asymptotic to the trajectory of $x_1$, then turns $Np_1$ times
around the first periodic orbit, follows the trajectory of
$x_2$ until it reaches $x_3$; then it turns $Np_3$ times around the
second periodic orbit, and so on, until it reaches $x_{2n}$ and goes
back to $x_1$, winding up on the trajectory of $x_1$. 
The resulting trajectory is in $T_{hyp}$ and,
repeating the argument from Lemma \ref{transitivity}, 
we see that it is nonwandering.

\medskip

Finally, we use the closing lemma on $\Omega_{NF}$ to
obtain a periodic orbit in $\Omega_{NF}$.
When $N$ is large, the time spent going from one periodic orbit to another
is small with respect to the time winding up around the periodic orbits,
so the Dirac measure on the resulting periodic orbit is close to the sum
$\sum_{i}c_{2i+1}\delta_{x_{2i+1}}$ and the theorem is proven.

\medskip


 \begin{center}
 \epsfig{figure=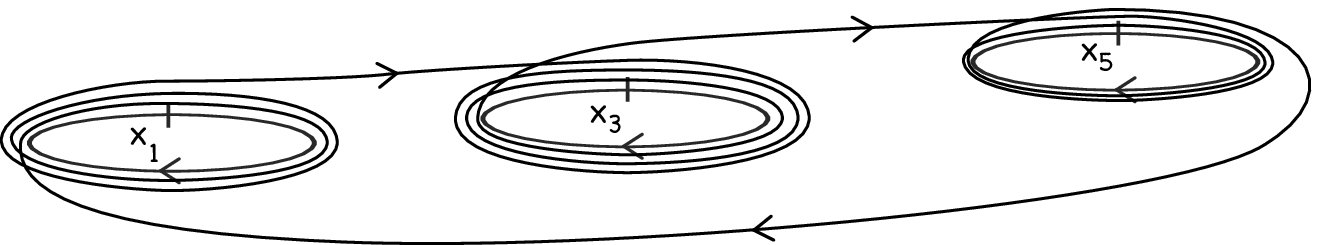,width=0.8\textwidth,angle=0}
 \end{center}

\end{proof}

The proof of theorem \ref{ergodicity-generic} is then straightforward
and follows verbatim from the arguments given in \cite{CS}.
We sketch the proof for the comfort of the reader.

\begin{proof}   
Proposition \ref{densite-orbites-periodiques}  ensures that 
ergodic measures are dense in the set of probability 
measures on $\Omega_{NF}$. 
The fact that they form a $G_\delta$-set is well known.

The fact that invariant measures of full support are a dense $G_\delta$-subset
of the set of invariant probability measures on $\Omega_{NF} $ is 
a simple corollary of the density of periodic orbits
in $\Omega_{NF}$, which itself follows from the closing lemma.

Finally, the intersection of two dense $G_\delta$-subsets 
of $\mathcal{M}^1(\Omega_{NF})$ is
still a dense $G_\delta$-subset of $\mathcal{M}^1(\Omega_{NF})$, 
because this set has the Baire property. 
This concludes the proof. 
\end{proof}


\subsection{Examples}\label{examples}

We now build examples for which the hypotheses or results
presented in that article do not hold.

\medskip

We start by an example of a surface for which $\Omega_{NF}$ is not open in
$\Omega$. First we consider a surface made up of 
an euclidean cylinder put on an euclidean plane. Such surface is built
by considering an horizontal line and a vertical line 
in the plane, and connecting them with a convex arc that is 
infinitesimally flat at its ends.
The profile thus obtained is then rotated along the vertical axis.
The negatively curved part is greyed in the figure below.


 \begin{center}
 \epsfig{figure=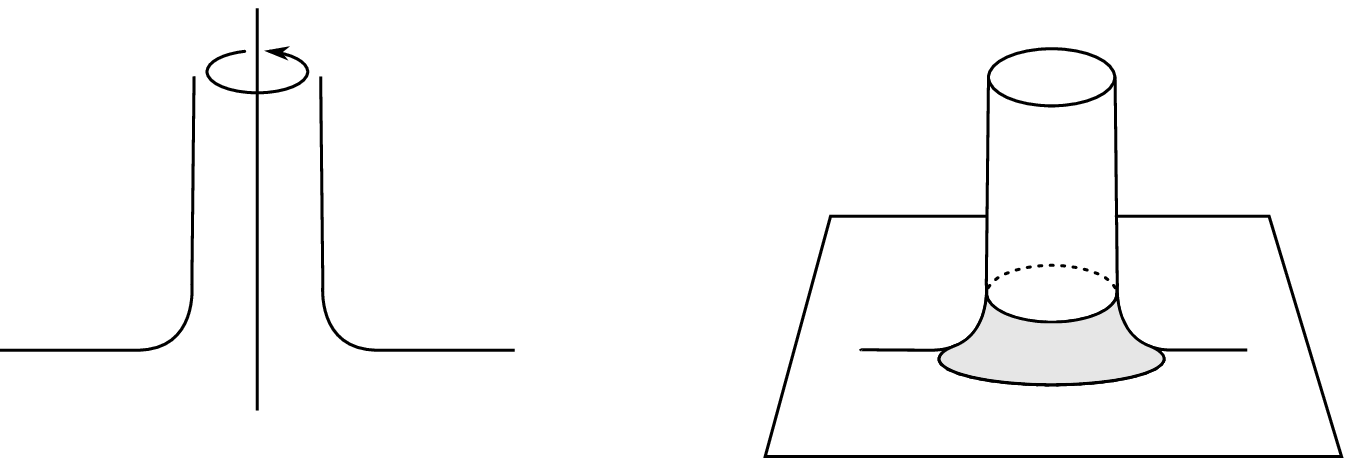,width=0.8\textwidth,angle=0}
 \end{center}

We can repeat that construction so as to line up cylinders
on a plane. Let us use cylinders of the same size and shape,
and take them equally spaced.
The quotient of that surface by the natural ${\bf Z}$-action
is a pair of pants, its three ends being euclidean flat cylinders.

 \begin{center}
 \epsfig{figure=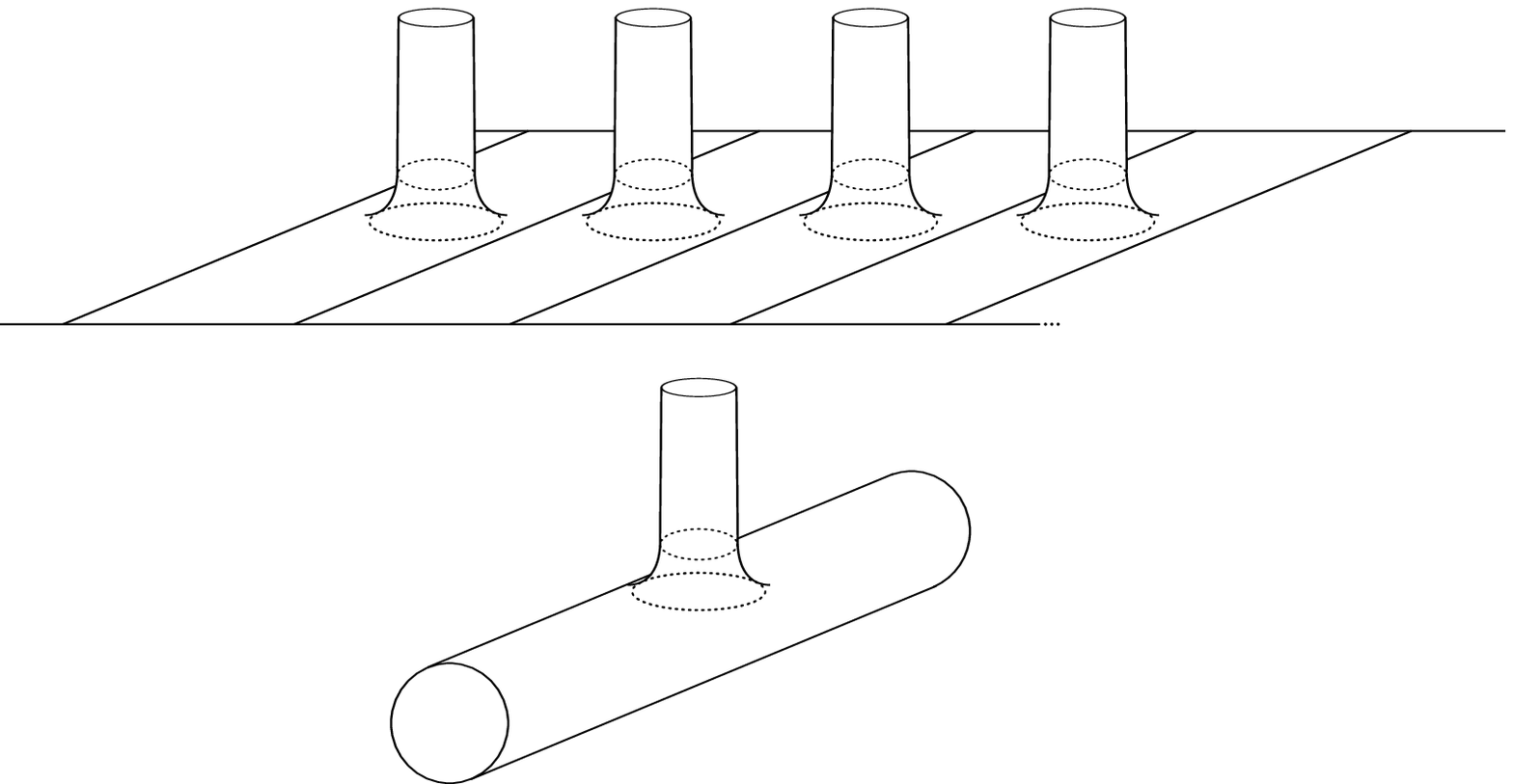,width=0.8\textwidth,angle=0}
 \end{center}

These cylinders are bounded by three closed geodesics that are 
accumulated by points of negative curvature.
The nonwandering set of the ${\bf Z}$-cover is the inverse image
of the nonwandering set of the pair of pants.
As a result, the lift of the three closed geodesics to the ${\bf Z}$-cover
are nonwandering geodesics. They are in fact accumulated by 
periodic geodesics turning around the cylinders a few times
in the negatively curved part, cf \cite{CS2}, th. 4.2 ff.
We end up with a row of cylinders on a strip bounded by two
nonwandering geodesics. These are the building blocks for our
example.

\medskip

We start from an euclidean half-plane and pile up alternatively
rows of cylinders with bounding geodesics $\gamma_i$ and $\gamma'_i$,
and euclidean flat strips. We choose the width so that
the total sum of the widths of all strips is converging.
We also increase the spacing between the cylinders from
one strip to another so as to insure that they do not accumulate 
on the surface.
The next picture is a top view of our surface, cylinders appear
as circles.

 \begin{center}
 \epsfig{figure=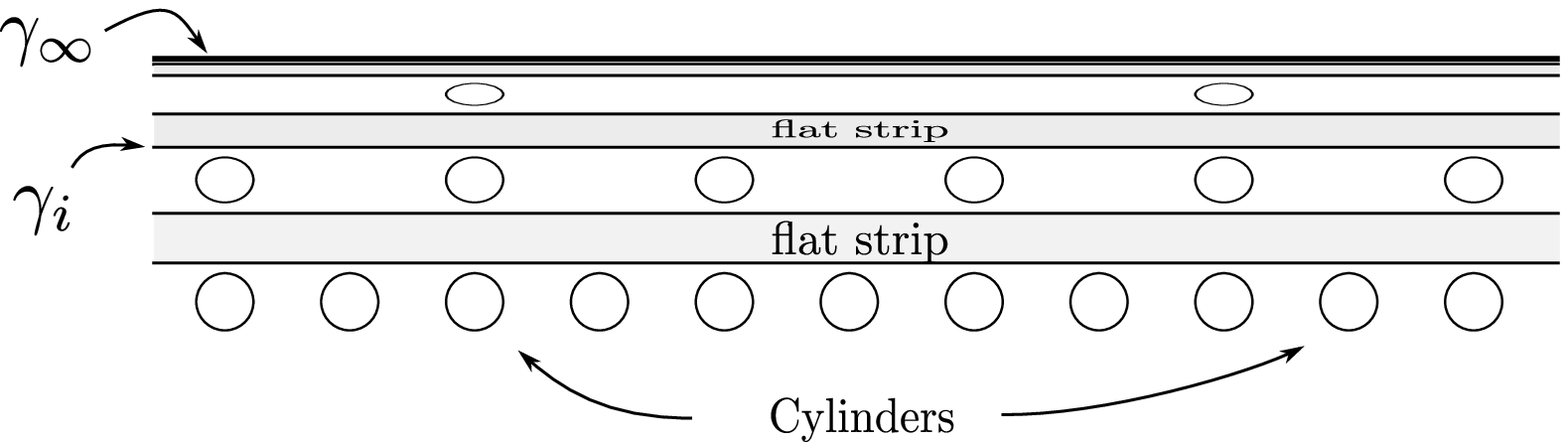,width=0.8\textwidth,angle=0}
 \end{center}

All the strips accumulate on a geodesic $\gamma_\infty$
that is nonwandering because it is in the closure of the periodic geodesics.
We can insure that it does not bound a flat strip by 
mirroring the construction on the other side of $\gamma_\infty$.
So $\gamma_\infty$ is in $\Omega_{NF}$, and is approximated by 
geodesics $\gamma_i$ that belong to $\Omega$ and bound a flat strip.
Thus, $\Omega_{NF}$ is not open in $\Omega$.
We conjecture  that ergodicity is a generic property in the set
of all probability measures invariant by the geodesic flow on that surface.  
The flat strips should not matter here
since they do not contain recurrent trajectories, but our method does not apply
to that example.

\bigskip

The next example, due to Gromov \cite{Gromov}, is detailed in
 \cite{Eberlein80} or \cite{Knieper}. Let $T_1$ be a torus with one
hole, whose boundary is homeomorphic to $S^1$, endowed with a nonpositively
curved metric, negative far from the boundary, and zero on a flat cylinder
homotopic to the boundary. Let $M_1=T_1\times S^1$. Similarly, let $T_2$ be the
image of $T_1$ under the symmetry with respect to a plane containing $\partial
T_1$, and $M_2=S^1\times T_2$. The manifolds $M_1$ and $M_2$ are
$3$-dimensional manifolds whose boundary is a euclidean torus. We glue them
along this boundary to get a closed manifold $M$ which contains around the
place of gluing a thickened flat torus, isometric to $[-r,r]\times
\mathbb{T}^2$, for some $r>0$.

\begin{figure}[ht!]
\begin{center}
\input{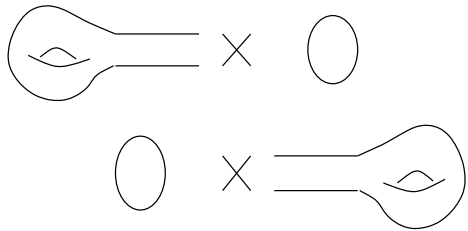}
\caption{Manifold containing a thickened torus}
\label{tore}
\end{center}
\end{figure}

Consider the flat $2$-dimensional torus $\{0\}\times \mathbb{T}^2$ embedded in
$M$. Choose an irrational direction $\{\theta\}$ on its unit tangent bundle
and lift the normalized Lebesgue measure of the flat torus to the invariant
set of unit tangent vectors pointing in this irrational direction $\theta$.
This measure is an ergodic invariant probability measure on $T^1M$, and the
argument given in \cite{CS2} shows that it is not in the closure of the set of
invariant ergodic probability measures of full support. In particular, ergodic
measures are not dense, and therefore not generic. Note also that this
measure is in the closure of the Dirac orbits supported by periodic orbits
bounding flat strips (we just approximate $\theta$ by a rational number),
but cannot be approximated by Dirac orbits
on periodic trajectories that do not bound flat strips.

This does not contradict our results though,
because this measure is supported in $\Omega\setminus\Omega_{NF}$ 
(which is closed).


\section{Measures with zero entropy}\label{section-entropy}

\subsection{Measure-theoretic entropy}

Let $X$ be a Polish space, $(\phi^t)_{t\in\R}$ a continuous flow on $X$, 
and $\mu$ a Borel invariant probability measure on $X$. 
As the  measure theoretic entropy satisfies the relation
 $h_{\mu}(\phi^t)=|t|h_{\mu}(\phi^1)$, 
we   define here the entropy of the application $T:=\phi^1$. 

\begin{defin} Let $\mathcal{P}=\{P_1,\dots, P_K\}$ 
be a finite partition of $X$ into Borel sets. 
The entropy of the partition $\mathcal{P}$ is the quantity 
$$ 
H_\mu(\mathcal{P})=-\sum_{P\in\mathcal{P}}\mu(P)\log \mu(P)\,.
$$
Denote by $\vee_{i=0}^{n-1}T^{-i}\mathcal{P}$ the finite
 partition into sets of the form 
$P_{i_1}\cap T^{-1}P_{i_2}\cap\dots \cap T^{-n+1}P_{i_n}$.
The {\em measure theoretic entropy of $T=\phi^1$ w.r.t. 
the partition $\mathcal{P}$} is defined by the limit
\begin{equation}\label{entropy}
h_{\mu}(\phi^1,\mathcal{P})=
\lim_{n\to\infty}\frac{1}{n}H_{\mu}(\vee_{i=0}^{n-1}T^{-i}\mathcal{P})\,.
\end{equation}
The {\em measure theoretic entropy of $T=\phi^1$} is 
defined as the supremum
$$
h_{\mu}(\phi^1)=\sup\{ h_\mu(\phi^1,\mathcal{P}),\,\,\mathcal{P}\mbox{ finite partition}\,\}$$
\end{defin}

The following result is classical \cite{wal}.

\begin{prop} Let $(\mathcal{P}_k)_{k\in\N}$
 be a increasing sequence of finite partitions of $X$ 
into Borel sets such that 
$\vee_{k=0}^\infty \mathcal{P}_k$  generates the Borel $\sigma$-algebra of $X$. 
Then the measure theoretic entropy of $\phi^1$ satisfies 
$$
h_{\mu}(\phi^1)=\sup_{k\in\N} h_{\mu}(\phi^1,\mathcal{P}_k)\,.
$$
\end{prop}


\subsection{Generic measures have zero entropy}

\begin{theo}\label{generic-zero-entropy} Let $M$ be a 
connected, complete, nonpositively curved manifold, 
whose geodesic flow admits at
least three different periodic orbits, that do not bound a flat strip.
Assume  that $\Omega_{NF}$ is open in $\Omega$. 
The set of invariant probability measures on $\Omega_{NF}$ with zero 
entropy is a dense $G_\delta$ subset of the
set $\mathcal{M}^1(\Omega_{NF} )$ of invariant probability
 measures supported in $\Omega_{NF} $. 
\end{theo}

Recall here that on a nonelementary negatively curved manifold, 
$\Omega=\Omega_{NF}$ so that the above theorem applies on the full
nonwandering set $\Omega$. 

The proof below is inspired from the proof of
Sigmund \cite{Sigmund}, who treated
the case of Axiom A flows on compact manifolds, and from results of
Abdenur, Bonatti, Crovisier \cite{ABC}
who considered nonuniformly hyperbolic diffeomorphisms on compact manifolds.
But no compactness assumption is needed in our statement.

\begin{proof} 
Remark first that on any Riemannian manifold $M$, 
if $B=B(x,r)$ is a small ball, $r>0$ being strictly less 
than the injectivity radius of $M$ at the point $x$, any geodesic 
(and in particular any periodic geodesic) intersects the boundary of $B$ 
in at most two points. 
Lift now the ball $B$ to the set $T^1B$ 
of unit tangent vectors of $T^1M$ with base points in $B$. 
Then the Dirac measure supported on 
any periodic geodesic intersecting $B$ 
gives zero measure to the boundary of $T^1B$. \\

Choose a countable family of balls $B_i=B(x_i,r_i)$, with centers dense in
$M$. Subdivide each lift $T^1B_i$ on the unit tangent bundle $T^1M$ into
finitely many balls, and denote by $(\mathcal{B}_j)$ the countable family of
subsets of $T^1M$ that we obtain. Any finite family of such sets
$\mathcal{B}_j$ induces a finite partition of $\Omega_{NFP}$ into Borel sets
(finite intersections of the $\mathcal{B}_j$'s, or their complements ). Denote
by $\mathcal{P}_k$ the finite partition induced by the finite family of sets
$(\mathcal{B}_j)_{0\le j\le k}$. If the family $\mathcal{B}_j$ is well chosen,
the increasing sequence $(\mathcal{P}_k)_{k\in\N}$ is such that
$\vee_{k=0}^\infty \mathcal{P}_k$ generates
the Borel $\sigma$-algebra. \\

Set $X=\Omega_{NF}$. According to proposition
\ref{densite-orbites-periodiques}, the family $\mathcal{D}$ of Dirac measures
supported on periodic orbits of $X$ is dense in $\mathcal{M}^1(X)$.
Denote by $\mathcal{M}^1_Z(X)$ the subset of probability measures
with entropy zero in $\mathcal{M}^1(X)$. 
The family $\mathcal{D}$ of Dirac measures
supported on periodic orbits of $X$ is included 
in $\mathcal{M}^1_Z(X)$, is dense in $\mathcal{M}^1(X)$, 
satisfies $\mu(\partial \mathcal{P}_k)=0$ and
$h_\mu(\mathcal{P}_k)=0$ for all $k\in\N$ and $\mu\in\mathcal{D}$. \\

Fix any $\mu_0\in\mathcal{D}$. 
Note that the limit in (\ref{entropy}) always exists, so 
that it can be replaced by a $\liminf$. 
As $\mu_0$ satisfies $\mu_0(\partial \mathcal{P}_k)=0$, 
if a sequence 
$\mu_i\in\mathcal{M}^1(X)$ converges in the 
weak topology to $\mu_0$, 
it satisfies for all $n\in\N$, 
$H_{\mu_i}(\vee_{j=0}^n g^{-j}\mathcal{P}_k)\to
 H_{\mu_0}(\vee_{j=0}^n g^{-j}\mathcal{P}_k)$
 when $i\to\infty$. 
In particular, the set 
$$
\{\mu\in\mathcal{M}^1(X),\,
H_{\mu}(\vee_{j=0}^n g^{-j}\mathcal{P}_k)<
H_{\mu_0}(\vee_{j=-n}^n g^j\mathcal{P}_k)+\frac{1}{r}\}\,,
$$ 
for $r\in\N^*$, is an open set. 
We deduce that  $\mathcal{M}_Z(X)$ is 
a $G_\delta$-subset of $\mathcal{M}(X)$.
 Indeed, 
\begin{eqnarray*}
\mathcal{M}^1_Z(X)
&=&\{\mu\in\mathcal{M}^1(X), \,h_\mu(g^1)=0=h_{\mu_0}(g^1)\}\\
&=&\cap_{k\in\N}\{\mu\in\mathcal{M}^1(X),\,h_{\mu}(g^1,\mathcal{P}_k)=0
=h_{\mu_0}(g^1,\mathcal{P}_k)\}\\
&=&\cap_{k\in\N}\cap_{r=1}^\infty\{\mu\in\mathcal{M}^1(X),\,0
\le h_{\mu}(g^1,\mathcal{P}_k)<\frac{1}{r}=h_{\mu_0}(g^1,\mathcal{P}_k)+\frac{1}{r}\}\\
&=&\cap_{k\in\N}\cap_{r=1}^\infty\cap_{m=1}^\infty\cup_{n=m}^\infty\\
&\, &\{\mu\in\mathcal{M}^1(X),\,\frac{1}{n+1}
H_{\mu}(\vee_{j=0}^n g^{-j}\mathcal{P}_k)<
\frac{1}{n+1}H_{\mu_0}(\vee_{j=0}^n g^{-j}\mathcal{P}_k)+\frac{1}{r}\,\}\,.
\end{eqnarray*}

The fact that $\mathcal{M}^1_Z(X)$ is dense is obvious 
because it contains the family $\mathcal{D}$ of periodic orbits of $X$. 
\end{proof}


\section{Mixing measures}\label{mixing}

\subsection{Topological mixing}

Let $(\phi^t)_{t\in\R}$ be a continuous flow on a Polish space $X$. 
The flow is said {\em topologically mixing} 
if for all open subsets $U,V$ of $X$, 
there exists $T>0$, such that for all $t\ge T$,
 $\phi^t U\cap V\neq \emptyset$. 
This property is of course stronger than transitivity:
 the flow is {\em transitive} if 
for all open subsets $U,V$ of $X$, 
and all $T>0$, there exists $t\ge T$, $\phi^t U\cap V\neq \emptyset$.
An invariant measure $\mu$ under the flow  
is {\em strongly mixing} if for all Borel sets $A$ and $B$ we have 
$\mu(A\cap \phi^t B)\to \mu(A)\mu(B)$ when $t\to +\infty$. 

\medskip

An invariant measure cannot be strongly mixing if the flow itself 
is not topologically mixing on its support (see e.g. \cite{wal}).
We recall therefore some results about topological mixing, 
which are classical on negatively curved manifolds, and still true here.

\begin{prop}[Ballmann, \cite{ballmann82}, rk 3.6 p. 54 
and cor. 1.4 p.45]
 Let $M$ be a connected rank one manifold, 
such that all tangent vectors are nonwandering ($\Omega=T^1M$). 
Then the geodesic flow is topologically mixing. 
\end{prop}

 Also related is the work of M. Babillot \cite{mbab} 
who obtained the mixing of the measure of maximal entropy under
suitable assumptions, with the help of a geometric cross ratio.

\begin{prop}\label{mixing-non-arithmeticity} 
Let $M$ be a connected, complete, nonpositively curved manifold, 
whose geodesic flow admits at least three distinct periodic orbits,
that do not bound a flat strip.
If $\Omega_{NF }$
is open in $\Omega$, then 
the restriction of the geodesic flow to $\Omega_{NF }$
is  topologically mixing iff the length spectrum 
of the geodesic flow restricted to $\Omega_{NF }$ 
is non arithmetic. 
\end{prop}

\begin{proof}  
Assume first that the geodesic flow restricted to $\Omega_{NF}$ 
is topologically mixing. The argument is classical. 
Let $u\in \Omega_{NF}$ be a vector, and $\varepsilon>0$. 
Let $\delta>0$ and $U\subset \Omega_{NF}$ 
be a neighbourhood of $u$ of the form $U=B(u,\delta)\cap\Omega_{NF}$ 
where the closing lemma is satisfied (see lemma \ref{weak-closing-lemma}). 

Topological mixing on $\Omega_{NF}$
implies that there exists $T>0$, 
s.t. for all $t\ge T$, $g^tU\cap U\neq\emptyset$.
Thus, for all $t\ge T$  there exists $v\in U\cap g^tU$, 
so that $d(g^tv,v)\le \delta$.

We can apply the closing lemma to $v$, 
and obtain  a periodic orbit of $\Omega_{NF}$ 
of length $t\pm\varepsilon$ shadowing the orbit of $v$ during the time $t$. 
As it is true for all $\varepsilon>0$ and large $t>0$,
 it implies the non arithmeticity of the length spectrum of the geodesic flow
in restriction to $\Omega_{NF}$. \\

We assume now that the length spectrum of the geodesic 
flow restricted to $\Omega_{NF}$ is non arithmetic
and we show that the geodesic flow is topologically mixing. 
In \cite{Dalbo}, she proves this implication on negatively 
curved manifolds, by using intermediate properties 
of the strong foliation. 
We give here a direct argument. 

$\bullet$ First, observe that it is enough to prove that for any
 open set $U\in \Omega_{NF}$, there exists $T>0$, 
such that for all $t\ge T$, $g^tU\cap U\neq \emptyset$. 
Indeed, if $U$, $V$ are two open sets of $\Omega_{NF}$, 
 by transitivity of the flow, there exists $u\in U$ 
and $T_0>0$ s.t. $g^{T_0}u\in V$. Now, by continuity of 
the geodesic flow, we can find a neighbourhood 
$U'$ of $u$ in $U$, such that $g^{T_0}(U')\subset V$. 
If we can prove that for all large $t>0$, 
$g^t(U')\cap U'\neq \emptyset$, we obtain that for 
all large $t>0$, $g^tU\cap V\neq\emptyset$. 

$\bullet$ Fix an open set $U\subset \Omega_{NF}$. 
Periodic  orbits of $\Omega_{hyp}$ are dense in $\Omega_{NF}$. 
Choose a periodic  orbit
$p\in U\cap \Omega_{hyp}$. As $U$ is open, there exists $\varepsilon>0$, 
such that $g^tp\in U$, for all $t\in [-3\varepsilon, 3\varepsilon]$. 
By non arithmeticity of the length spectrum, 
there exists another periodic vector $p_0\in \Omega_{hyp}$, and positive integers 
$n,m\in \Z$, $|nl(p)-ml(p_0)|<\varepsilon$. 
Assume that $0<nl(p)-ml(p_0)<\varepsilon$. 

$\bullet$ By transitivity of the geodesic flow on $\Omega_{NF}$, 
and local product  
choose a vector $v$ negatively asymptotic to the negative geodesic
orbit of  $p$ and positively asymptotic to the geodesic orbit of $p_0$, 
and a vector $w$ negatively asymptotic to the orbit of $p_0$ 
and positively asymptotic 
to the orbit of $p$. By lemma \ref{weak-local-product} (2), 
$v$ and $w$ are in $T_{hyp}$. Moreover, 
they are nonwandering by the same argument as in the proof
 of lemma \ref{transitivity}. 
Using the local product structure and the closing lemma, we can construct 
for all
positive integers $k_1,k_2\in\N^*$ a periodic vector $p_{k_1,k_2}$ at distance 
less than $\varepsilon$ of $p$, whose orbit turns $k_1$ 
times around the orbit of $p$, going from an $\varepsilon$-neighbourhood 
of $p$ to an $\varepsilon$-neigbourhood of $p_0$, 
with a ``travel time'' $\tau_1>0$, turning around the orbit of $p_0$
 $k_2$ times, and coming back to the $\varepsilon$-neighbourhood of $p$, 
with a travel time $\tau_2$. 
Moreover, $\tau_1$ and $\tau_2$ are independent of $k_1,k_2$ 
and depend only on $\varepsilon$, and on the initial choice of $v$ and $w$. 
The period  of $p_{k_1,k_2}$ is 
$k_1l(p)+k_2l(p_0)+C(\tau_1,\tau_2,\varepsilon)$,
 where $C$ is a constant, and $g^\tau p_{k_1,k_2}$ belongs to $U$ 
for all $\tau\in]-\varepsilon,\varepsilon[$. 

$\bullet$ Now, by non arithmeticity, there exists $T>0$ large enough, s.t. 
the set 
$\{k_1l(p)+k_2l(p_0)+C(\tau_1,\tau_2,\varepsilon),\,\, k_1\in\N,\,k_2\in\N\,\}$ 
is $\varepsilon$-dense in $[T,+\infty[$. 
To check it, let $K_0$ be the largest integer such 
that $K_0(nl(p)-ml(p_0))<ml(p_0)$. 
Observe then that for all positive integer $i\ge 1$, 
and all $0\le j\le K_0+1$, the set of points 
$(K_0+i)ml(p_0)+j(nl(p)-ml(p_0))=(K_0+i-j)ml(p_0)+jnl(p)$  
is $\varepsilon$-dense in $[(K_0+i)ml(p_0),(K_0+i+1)ml(p_0)]$.

As $g^\tau p_{k_1,k_2}$ belongs to $U$ for all $\tau\in]-\varepsilon,\varepsilon[$, 
it proves that for all $t\ge T$, $g^tU\cap U\neq \emptyset$. 
 \end{proof}

\subsection{Strong mixing}

Even in the case of a topologically mixing flow, 
generic measures are not strongly mixing, 
according to the following result.

\begin{theo}\label{non-mixing-generic}
Let $(\phi^t)_{t\in\R}$ be a continuous flow on a 
complete separable metric space $X$. 
If the Dirac measures supported by periodic orbits 
are dense in the set of invariant probability measures on $X$, 
 then the set of invariant 
measures which are {\em not} strongly mixing contains 
a dense $G_\delta$-subset of the set
 of invariant probability measures on $X$. 
\end{theo}

This result was first proven by K. R. Parthasarathy in the context of discrete
symbolic dynamical systems \cite{Pa1}.  
We adapt here the argument in the setting of flows. 
Thanks to  proposition \ref{densite-orbites-periodiques}   we obtain:

\begin{corol} Let $M$ be a complete, connected, nonpositively curved manifold 
with at least three different periodic orbits, that do not bound a 
flat strip. If $\Omega_{NF}$ is open in $\Omega$, then 
the set of invariant measures which are {\em not} strongly mixing contains 
a dense $G_\delta$-subset of the set of invariant
probability measures on $\Omega$.
\end{corol}

\begin{proof} 
Choose a countable dense set of points $\{x_i\}$, and let
$\mathcal{A}$ be the countable family of all closed balls of rational radius %
centered at a point $x_i$. This family generates the Borel $\sigma-$algebra
of $T^1M$. A measure $\mu$ is a strongly mixing measure 
if for any set $F\in\mathcal{A}$ such that 
$\mu(F)>0$, we have $\mu(F\cap \phi^t F)\to \mu(F)^2$ when $t\to\infty$.

\medskip

\noindent
For any subset $F_1\in \mathcal{A}$, 
let $G_n=V_{\frac{1}{n}}(F_1)$ be a decreasing sequence 
of open neighbourhoods of $F_1$ with intersection $F_1$. 
The set of strongly mixing measures is included in the following union 
(where all indices $n, \varepsilon, \eta, r$ are rational numbers, 
$t$ is a real number and $F_1,F_2$ are disjoint)

$$ \bigcup_{F_1 ,F_2\in\mathcal{A}}\ 
   \bigcup_{n\in\N^*}\ 
   \bigcup_{\varepsilon\in(0,1)}\ 
   \bigcup_{0<\eta<2\varepsilon^2/3}\
   \bigcup_{r\in(0,1)}\ 
   \bigcup_{m\in\N}\ 
   \bigcap_{t\ge m}\\
   A_{F_1,F_2,n,\varepsilon,\eta,r,m,t}
$$
with $A_{F_1,F_2,n,\varepsilon,\eta,r,m,t} \subset \mathcal{M}(X)$ given by
$$
\{\mu\in\mathcal{M}(X)\mid \mu(F_1)\ge\varepsilon,\,\mu(F_2)\ge\varepsilon,\,\,
\mu(G_n\cap \phi^k G_n)\le r, \,r\le \mu^2(F_1)+\eta\}\,.
$$
This set is closed, because $G_n$ is an open set, and
$F_1,F_2$ are closed. (The second closed set $F_2$ is disjoint from $F_1$ and
is just used to guarantee that $F_1$ is not of full measure). The intersection
of all such sets over all $t\ge m $ is still closed.
The set of strongly mixing measures is therefore 
included in a countable union of closed sets. 

\medskip

Let us show that each of these closed sets has empty interior. 
Denote by $\mathcal{E}(F_1,F_2,G_n,\varepsilon,r,m)$ the set
$$
\bigcap_{t\geq m}\ 
\{\mu\in\mathcal{M}(X) \mid \mu(F_1)\ge\varepsilon,\,\mu(F_2)\ge\varepsilon,\,
\mu(G_n\cap \phi^t G_n)\le r, \,r\le \mu^2(F_1)+\eta\}.
$$
It is enough to show that its complement contains all periodic measures.
Remark first that if $\mu$ is a Dirac measure supported on a periodic orbit
of length $l$, then for all Borel sets $A\subset X$, and all multiples $jl$ of
the period,
$$
\mu(A\cap \phi^{jl} A)=\mu(A)\,.
$$
In particular, they are obviously not mixing.

\medskip

\noindent
Let $\mu_0$ be a periodic measure of period $l>0$, and $j\ge 1$ an integer
s.t. $jl\ge m$. Let us show that it does not belong to the following set:
$$
\{\mu\in\mathcal{M}(X),\,\mu(F_1)\ge\varepsilon,\, \mu(F_2)\ge\varepsilon,\,
\mu(G_n\cap \phi^{jl} G_n)\le r, \,r\le \mu^2(F_1)+\eta\}.
$$ 
If $\mu_0(F_1)\ge \varepsilon$ and $\mu_0(F_2)\ge \varepsilon$, 
we get $\varepsilon\le \mu_0(F_1)\le 1-\varepsilon$. 
The key property of $\mu_0$ gives $\mu_0(G_n\cap\phi^{jl} G_n)=\mu_0(G_n)$. 
We deduce that
\begin{eqnarray*}
\mu_0(G_n\cap \phi^{jl}G_n)-\mu_0(F_1)^2&=&
\mu_0(G_n)-\mu_0(F_1)^2\ge\mu_0(F_1)(1-\mu_0(F_1)\ge
\varepsilon(1-\varepsilon)\\
&\ge&\varepsilon^2>\frac{3\eta}{2}>\eta
\end{eqnarray*}
so that $\mu_0$ does not belong to the above set. 
In particular, the
 periodic measures do not belong to 
$\mathcal{E}(F_1,F_2,G_n,\varepsilon,r,m)$ and the result is proven. 
\end{proof}


\subsection{Weak mixing}

We end with a question concerning the weak-mixing property.
An invariant measure $\mu$ on $X$ is {\em weakly mixing} if for all 
continuous function with compact support $f$ defined on $X$, we have 
\begin{equation}\label{weakmixing}
\lim_{T\to\infty}\frac{1}{T}\int_0^T
\left|\int_X f\circ\phi^t(x)\,f(x)\,d\mu(x)-\left(\int_X f\,d\mu\right)^2\right|
\,dt =0\,.
\end{equation}

\begin{theo}[Parthasarathy, \cite{Pa2}] Let $(\phi^t)_{t\in\R}$ 
be a continuous flow 
on a Polish space. The set of weakly mixing measures on $X$ 
is a $G_\delta$-subset of the set 
of Borel invariant probability measures on $X$. 
\end{theo}

\noindent
Of course, this result applies in our context, 
with $X=\Omega$,  or $X=\Omega_{NF}$. 

In the case of the dynamics on a full shift, 
Parthasarathy proved 
in \cite{Pa2} that there exists a dense subset of strongly mixing measures. 
This result was improved by Sigmund \cite{si} 
who showed that there is a dense subset of Bernoulli measures.
 Of course, these results  imply in particular 
that the above $G_\delta$-set is 
a dense $G_\delta$-subset of $\mathcal{M}(\Omega)$. 
But the methods of \cite{Pa2} and \cite{si} strongly use specific 
properties of  a shift dynamics, 
and seem therefore difficult to generalize. 
In any case, such a result would impose 
to add the assumption that the flow is 
topologically mixing. 

Anyway, the following question is interesting: in the setting of noncompact
rank one manifold, can we find a dense family of weakly mixing measures on
$\Omega_{NF}$ ? Or at least one ? \medskip

We recall briefly the proof of the above theorem for the reader. The arguments
are similar to those of \cite{Pa2}, but our formulation is shorter.

\begin{proof} It is classical that the weak mixing of 
the system $(X,\phi,\mu)$ is equivalent to 
the ergodicity of  $(X\times X, \phi\times\phi, \mu\times\mu)$ 
(see e.g. \cite{wal}).

\medskip

Let $(f_i)_{i\in\N}$ be some countable algebra of Lipschitz bounded functions on
$X\times X$ separating points. Such a family is dense in the set of all bounded
Borel functions, with respect to the $L^2(m)$ norm, for all Borel probability
measures $m$ on $X\times X$ (see \cite{cou3} for a short proof). Now, the
complement of the set of weakly mixing measures $\mu\in\mathcal{M}(X)$
can be written as the union of the following sets:
\begin{eqnarray*}
F_{k,l,i}=\{\mu\in\mathcal{M}(X),\,
\exists \, m_1,m_2\in\mathcal{M}(X\times X), 
\alpha\in[\frac{1}{k},1-\frac{1}{k}], \,s.t.\\
\mu\times\mu=\alpha m_1+(1-\alpha)m_2,\mbox{ and }
\int f_i dm_1\ge\int f_i dm_2+\frac{1}{l}\}\,.
\end{eqnarray*}
We check as in \cite{CS} that these sets are closed, 
so that the weakly mixing measures of $X$ form 
a $G_\delta$-subset of $\mathcal{M}(X)$. 
\end{proof}


\end{document}